\documentclass[a4paper,10pt,reqno]{amsart}

\usepackage{amsthm,amsmath,amssymb}
\usepackage{mathrsfs}
\usepackage{stmaryrd}
\usepackage[raise]{fmtcount}
\usepackage{nicefrac}

\usepackage{xcolor}
\usepackage{hyperref}

\definecolor{blue-url}{RGB}{0,0,100}

\hypersetup{
	pdftitle={On the series of reciprocals of primes},
	pdfauthor={Salvatore Tringali},
	pdfmenubar=false,
	pdffitwindow=true,
	pdfstartview=FitH,
	colorlinks=true,
	linkcolor=blue-url,
	citecolor=blue-url,
	urlcolor=blue-url
}

\newtheorem{theorem}{Theorem}
\newtheorem{corollary}[theorem]{Corollary}

\newcommand{\ts}{\textsuperscript}

\begin{document}

\title[A generalization of Euler's theorem on the series of reciprocals of primes]{A generalization to number fields \\ of Euler's theorem on the series of reciprocals of primes}
\author{Salvatore Tringali}
\address{College of Mathematics and Information Science, Hebei Normal University | Shijiazhuang 050024, China}
\email{salvo.tringali@gmail.com}
\urladdr{http://imsc.uni-graz.at/tringali}

\subjclass[2010]{Primary 11R04, 11R27 Secondary 11R42, 13F15}

\keywords{Atoms, Euler, irreducible elements, number fields, primes, rings of integers, series of reciprocals, zeta functions.}

\begin{abstract}
	Let $X$ be a set of positive integers, and let $\mathbb Z_K$ be the ring of integers of a number field $K$ of degree $n$. Denote by $N(I)$ the absolute norm of an ideal $I$ of $\mathbb Z_K$, and by $\mathcal A$ the set of principal ideals $a\mathbb Z_K$ such that $a$ is an atom of $\mathbb Z_K$ and $a$ divides $m$ for some $m \in X$. 
	
	Building upon the ideas of Clarkson from [Proc. Amer. Math. Soc. \textbf{17} (1966), 541], we show that, if the series $\sum_{m \in X} 1/m$ diverges, then so does the series $\sum_{\mathfrak a \in \mathcal A} |N(\mathfrak a)|^{-\nicefrac{1}{n}}$.
	Most notably, this generalizes a classical theorem of Euler on the series of reciprocals of positive rational primes.
\end{abstract}
\maketitle
\thispagestyle{empty}

\section{Introduction}
\label{sec:intro}
Let $\mathbb P$ be the set of positive rational primes. It is a classical result of Euler, hereinafter simply referred to as Euler's theorem, that the series
$\sum_{p \in \mathbb P} 1/p$ diverges, a short proof of which was given by Clarkson \cite{Clark66}. (Notation and terminology, if not explained when first used, are standard or should be clear from the context.)

The goal of the present paper is primarily to generalize Euler's theorem to the ring of integers of a number field, by an elementary approach that is essentially inspired to Clarkson's proof. 

\section{Main result}
Let $\mathbb Z_K$ be the ring of integers of a number field $K$,
and let $\mathcal I$ be the set of all non-zero ideals of $\mathbb Z_K$. For an ideal $\mathfrak i \in \mathcal I$ we denote by $N(\mathfrak i)$ the \emph{absolute norm} of $\mathfrak i$, that is, the cardinality of the quotient ring $\mathbb Z_K/\mathfrak i$. It is a basic result in number theory that $\mathbb Z_K$ is a Dedekind domain with the \emph{\textup{(FN)}-property}, i.e., such that $N(\mathfrak i) < \infty$ for every $\mathfrak i \in \mathcal I$, see e.g. \cite[p. 43 and Corollary to Theorem 1.20, p. 16]{Nar04}. This implies that, for each $\kappa \in \mathbb R^+$, there are only finitely many ideals $\mathfrak i \in \mathcal I$ with $N(\mathfrak i) \le \kappa$, see e.g. \cite[Theorem 1.6(ii)]{Nar04}. Consequently, it makes sense, for every $\mathcal J \subseteq \mathcal I$, to introduce the function
\[
\zeta_{K,\mathcal{J}}: \mathbb R \to \mathbb R \cup \{\infty\}: s \mapsto \sum_{\mathfrak i \in \mathcal J} \frac{1}{N(\mathfrak i)^s}
\]
and investigate its properties, as similarly done in the study of Dedekind zeta functions (the case $\mathcal J = \mathcal I$). Here as usual, we adopt the convention that
\[
\sum_{\mathfrak i \in \mathcal J} := \lim_{\kappa \to \infty}\sum_{\mathfrak i \in \mathcal J\,:\, N(\mathfrak i) \le \kappa}.
\]
In particular, let a \emph{unit} of $\mathbb Z_K$ be an element $u \in \mathbb Z_K$ such that $uv = 1_K$ for some $v \in \mathbb Z_K$, and an \emph{atom} (or \emph{irreducible element}) of $\mathbb Z_K$ be a non-unit element $a \in \mathbb Z_K$ such that $a \ne xy$ for all non-unit elements $x, y \in \mathbb Z_K$. We call two elements $x, y \in \mathbb Z_K$ \emph{associated} if $x = uy$ for some unit $u \in \mathbb Z_K$, and we write $y \mid_{\mathbb Z_K} x$ to mean that $y \in \mathbb Z_K$ and $x \in y \mathbb Z_K$.

With these preliminaries in place, we have the following theorem (the main contribution of the paper).

\begin{theorem}\label{thm:1}
	Let $K$ be a number field of degree $n$, let $X$ be a set of positive integers such that $\sum_{m \in X} 1/m = \infty$, and let $\mathcal A$ be the set of principal ideals $a\mathbb Z_K$ of $\mathbb Z_K$ for which $a$ is an atom and $a \mid_{\mathbb Z_K} m$ for some $m \in X$. Then $\zeta_{K,\mathcal{A}}(1/n) = \infty$, and hence $\mathcal A$ is infinite.
\end{theorem}

\begin{proof}
	To begin, we recall that, if $\mathfrak i$ is a non-zero principal of $\mathbb Z_K$ and $\alpha$ is a generator of $\mathfrak i$, then 
	\begin{equation}\label{equ:(1)}
	N(\mathfrak i) = |\mathrm{N}_{K/\mathbb Q}(\alpha)| = f_\alpha(0)^{n/\deg f_\alpha},
	\end{equation}
	where  $\mathrm{N}_{K/\mathbb Q}$ is the norm of $\alpha$ relative to the field extension $K/\mathbb Q$ and $f_\alpha$ is the minimal polynomial of $\alpha$ over $K$, see e.g. \cite[p. 48 and Corollary to Proposition 2.13, p. 58]{Nar04}. 
	
	This said, assume for a contradiction that $\zeta_{K,\mathcal{A}}(1/n) < \infty$. Then there exists $\kappa_0 \in \mathbb R^+$ such that
	\begin{equation*}
	\sum_{\mathfrak a \in \mathcal A\,:\, N(\mathfrak a) \ge \kappa_0} \frac{1}{N(\mathfrak a)^{1/n}} := \lim_{\kappa \to \infty} \sum_{\mathfrak a \in \mathcal A\,:\, \kappa_0 \le N(\mathfrak a) \le \kappa} \frac{1}{N(\mathfrak a)^{1/n}} < 1.
	\end{equation*}
	By the (FN)-property of $\mathbb Z_K$, we have that the set 
	\[
	\mathcal B := \{\mathfrak a \in \mathcal A: N(\mathfrak a) < \kappa_0\}
	\] 
	is finite. Moreover, it is seen from \eqref{equ:(1)} and \cite[Proposition 3.10]{Nar04} that the absolute norm of an ideal in $\mathcal A$ is greater than or equal to $2$. So putting everything together, we find that
	\begin{equation}\label{equ:(2)}
	\left(\prod_{\mathfrak b \in \mathcal B} \sum_{j=0}^\infty \frac{1}{N(\mathfrak b)^{j/n}}\right) \cdot \sum_{i=0}^\infty \left(\sum_{\mathfrak a \in \mathcal A\setminus \mathcal B} \frac{1}{N(\mathfrak a)^{1/n}}\right)^{\!i} < \infty.
	\end{equation}
	We will however prove that the expansion of the left-hand side of the latter inequality contains the term $1/m$ for every $m \in X_{\ge 2}$, contradicting that $\sum_{m \in X} 1/m = \infty$.
	For, fix $m \in X_{\ge 2}$.
	Then \eqref{equ:(1)} implies that
	\begin{equation}\label{equ:(3)}
	N(m \mathbb Z_K)^{1/n} = |\mathrm{N}_{K/\mathbb{Q}}(m)|^{1/n} = m,
	\end{equation}
	which, in view of \cite[Proposition 3.10]{Nar04}, shows that $m$ is not a unit of $\mathbb Z_K$. But $\mathbb Z_K$, being a Dedekind domain, is Noetherian, see e.g. \cite[Theorem 1.5]{Nar04}; and Noetherian domains are atomic, that is, every non-zero non-unit element in $\mathbb Z_K$ is a finite non-empty product of atoms (although such a factorization need not be unique), see e.g. \cite[Example 1.1.5.2]{GeHK06}. As a consequence there exist a unit $u \in \mathbb Z_K$, pairwise non-associated atoms $a_1, \ldots, a_\ell \in \mathbb Z_K$ ($\ell \in \mathbb N^+$), and exponents $e_1, \ldots, e_\ell \in \mathbb N^+$ such that $m = a_1^{e_1} \cdots a_\ell^{e_\ell} u$. Because $N$ is a totally multiplicative function, see e.g. \cite[Theorem 1.16(i)]{Nar04}, it thus follows from \eqref{equ:(3)} that
	\begin{equation}\label{equ:(4)}
	m = N(a_1 \mathbb Z_K)^{e_1/n} \cdots N(a_\ell \mathbb Z_K)^{e_\ell/n} N(u\mathbb Z_K) = N(a_1 \mathbb Z_K)^{e_1/n} \cdots N(a_\ell \mathbb Z_K)^{e_\ell/n}. 
	\end{equation}
	In fact, we can assume without loss of generality that there is a non-negative integer $k$ not greater than $\ell$ such that $a_i \mathbb Z_K \in \mathcal A \setminus \mathcal B$ for $1 \le i \le k$ and $a_i \mathbb Z_K \in \mathcal B$ for $k < i \le \ell$ (note that the atoms $a_1, \ldots, a_\ell$ are all in $\mathcal A$). Then it is clear from \eqref{equ:(4)} that $1/m$ occurs as a term in the expansion of
	\[
	\left(\prod_{k < j \le \ell} \frac{1}{N(a_j\mathbb Z_K)^{e_j/n}}\right) \cdot \left(\sum_{\mathfrak a \in \mathcal A \setminus \mathcal B} \frac{1}{N(\mathfrak a)^{1/n}}\right)^{\! \max_{1 \le i \le k} e_i},
	\]
	which is in turn one of the summands contributing to the left-hand side of inequality \eqref{equ:(2)}. 
	
	This is enough to complete the proof, because it is obvious that, if $\mathcal A$ were a finite set, then the series $\sum_{\mathfrak a \in \mathcal A} |N(\mathfrak a)|^s$ would be convergent for every $s \in \mathbb R$.
\end{proof}

Finally we prove Euler's theorem and two of a number of corollaries naturally stemming from Theorem \ref{thm:1} (we continue to use $\mathbb P$ for the set of positive rational primes, as in \S{ }\ref{sec:intro}).

\begin{corollary}\label{cor:1}
	Let $K$ be a number field of degree $n$, and let $\mathcal A$ be the set of principal ideals of $\mathbb Z_K$ generated by an atom. Then $\zeta_{K,\mathcal{A}}(1/n) = \infty$, and $\mathbb Z_K$ has infinitely many pairwise non-associated atoms.
\end{corollary}

\begin{proof}
	Straightforward from Theorem \ref{thm:1} and the fact that the harmonic series is divergent.
\end{proof}

\begin{theorem}[Euler's theorem]\label{thm:2}
	The series $\sum_{p \in \mathbb P} 1/p$ diverges.
\end{theorem}

\begin{proof}
The ring $\mathbb Z$ of ``ordinary'' integers is the ring of integers of the rational field $\mathbb Q$, and it is well known that, since $\mathbb Z$ is a PID, the atoms of $\mathbb Z$ are precisely the primes in $\mathbb P$ and their opposites, see e.g. \cite[p. 31 and Lemma 3.34]{Sha00}.
Hence, the set $\mathcal A$ of principal ideals of $\mathbb Z$ generated by an atom is the set $\{p\mathbb Z: p \in \mathbb P\}$, and we conclude at once from Corollary \ref{cor:1} that $\sum_{p \in \mathbb P} 1/p = \infty$, because $\mathbb Q$ is a number field of degree $1$ and $N(m\mathbb Z) = m$ for every positive integer $m$.
\end{proof}

\begin{corollary}\label{cor:3}
	Let $K$ be a number field of degree $n$, and let $\mathcal A$ be the set of principal ideals $a\mathbb Z_K$ of $\mathbb Z_K$ such that $a$ is an atom and $a \mid_{\mathbb Z_K} p$ for some $p \in \mathbb P$. Then $\zeta_{K,\mathcal{A}}(1/n) = \infty$.
\end{corollary}

\begin{proof}
	Straightforward from Theorems \ref{thm:1} and \ref{thm:2}.
\end{proof}

\section{Closing remarks}

1. Euler's theorem is commonly generalized to the ring of integers of a number field $K$ in a different way that done in this paper. One starts by considering the set $\mathcal P$ of prime ideals of $\mathbb Z_K$ and then proves that $\zeta_{K,\mathcal{P}}(1) = \infty$. This follows from the analogue of Euler's product formula for the Dedekind function $\zeta_K$ of $K$, combined with the fact that $\zeta_K$ has a pole at $1$, see e.g. \cite[Proposition 7.2 and Theorem 7.3]{Nar04}.
If $K$ is the rational field, then $\mathcal P = \{p \mathbb Z: p \in \mathbb P\}$ and $N(p \mathbb Z) = p$ for every $p \in \mathbb P$. So, similarly as in the proof of Corollary \ref{cor:2}, one recovers Euler's theorem.

2. That $\mathbb Z_K$ has infinitely many pairwise non-associated atoms (Corollary \ref{cor:1}) can be proved by a more standard approach.
For, let $R$ be an atomic integral domain containing infinitely many pairwise coprime non-unit elements (two elements $a, b \in R$ are coprime if $aR + bR = R$), and
note that these conditions are satisfied by $\mathbb Z_K$ (see the proof of Theorem \ref{thm:1} and recall that, by B\'ezout's identity, $a\mathbb Z + b \mathbb Z = \mathbb Z$ for all  $a, b \in \mathbb Z$ whose greatest common divisor is $1$). Every non-unit element in $R$ is contained in a maximal ideal of $R$, and each maximal ideal of $R$ is prime, see e.g. \cite[Corollary 3.10 and Remark 3.25(i)]{Sha00}.
But coprime elements $a, b \in R$ cannot belong to the same maximal ideal $\mathfrak m$, or else $R = aR + bR \subseteq \mathfrak m \subsetneq R$. Hence, $R$ has countably infinitely many ideals $\mathfrak p_1, \mathfrak p_2, \ldots$ that are both prime and maximal. It follows by the prime avoidance theorem, see e.g. \cite[Theorem 3.61 and Remark 3.62(i)]{Sha00}, that for each $i \in \mathbb N^+$ there exists a non-unit element $x_i \in \mathfrak p_i$ that is not in $\mathfrak p_j$ for every $j \in \mathbb N^+$ with $i \ne j$. So, using that $R$ is atomic (by hypothesis) and a product of finitely many elements of $R$ is contained in a prime ideal of $R$ only if at least one of such elements is in $R$ (essentially by definition), let $a_i$ be an atom of $R$ in $\mathfrak p_i$ such that $x_i \in a_i R$ ($i \in \mathbb N^+$), and observe that $a_i \notin \mathfrak p_j$ for all $j \in \mathbb N^+$ with $i \ne j$ (otherwise $x_i \in a_i R \subseteq \mathfrak p_j$). This implies that $a_1, a_2, \ldots$ is a sequence of pairwise non-associated atoms of $R$, and we are done.

3. Let $\mathcal A$ be as in Corollary \ref{cor:1}, and for each $n \in \mathbb N$ define $a_n := |\{\mathfrak a \in \mathcal A: N(\mathfrak a) = n \}|$.
It is clear that
\begin{equation}\label{equ:dirichlet-series}
\zeta_{K, \mathcal A}(s) = \sum_{\mathfrak a \in \mathcal A} \frac{1}{N(\mathfrak a)^s} = \sum_{n \ge 1} \frac{a_n}{n^s}, \quad (s \in \mathbb R).
\end{equation}
On the other hand, if $D$ denotes the Davenport constant of the ideal class group of $K$ (see \cite[\S{ }9.1.3]{Nar04} for terminology), then we have by \cite[Theorem 9.15]{Nar04} that there exists a real constant $C > 0$ such that 
\[
A_n(x) := \sum_{n \le x} a_n \sim C \frac{x}{\log x} \,(\log \log x)^{D-1}, \quad \text{as }x \to \infty.
\]
It thus follows from \cite[Theorem 7]{HarRie15} that the abscissa of convergence of the Dirichlet series $\sum_{n \ge 1} a_n/n^s$ is $1$, which, together with \eqref{equ:dirichlet-series}, shows that $\zeta_{K, \mathcal A}(s) = \infty$ for every real $s < 1$, which is much stronger than Corollary \ref{cor:1}. Note, however, that the ideas behind this stronger result lie deep in the analytic theory of number fields and, to the best of the author's knowledge, fall short of proving anything like Corollary \ref{cor:3}.

\end{document}